\documentclass[leqno,11pt,oneside]{amsart}

 \usepackage[english]{babel}

\usepackage{xcolor}

\usepackage{graphicx}

\usepackage{enumitem}
\usepackage{amsmath,amsfonts,amssymb,amsthm}
\usepackage{mathrsfs}
\usepackage{relsize}

\usepackage[english]{babel}

\newtheorem*{theorem*}{Theorem}

\newtheorem{teo}{Theorem}[section]
\newtheorem{prop}[teo]{Proposition}
\newtheorem{ddef}[teo]{Definition}
\newtheorem{example}[teo]{Example}
\newtheorem{cor}[teo]{Corollary}

\newtheorem{lem}[teo]{Lemma}
\newtheorem*{cor*}{Corollary}

\newtheorem*{lem*}{Lemma}

\newtheorem*{teorA'}{Theorem A'}

\newtheorem*{fact*}{Fact}

\theoremstyle{definition}

\newcommand{\dd}{{\rm d}}

\newcommand{\dr}{\partial}

\newcommand{\C}{\mathbb{C}}

\newcommand{\R}{\mathbb{R}}

\newcommand{\Z}{\mathbb{Z}}
\newcommand{\Q}{\mathbb{Q}}

\newcommand{\A}{\mathcal{A}}

\newcommand{\D}{\mathcal{D}}

\newcommand{\F}{\mathcal{F}}

\newcommand{\G}{\mathcal{G}}
\newcommand{\LL}{{\mathcal L}}

\newcommand{\im}{{\rm Im}}
\newcommand{\re}{{\rm Re}}
\newcommand{\cl}[1]{\mbox{$\mathcal{#1}$}}
\newcommand{\ov}[1]{\mbox{$\overline{#1}$}}

\newcommand{\diff}{{\rm Diff}}

\newcommand{\holv}{{\rm Hol}^{\, vir}}

\newcommand*\xbar[1]{ %
   \hbox{ %
     \vbox{%
       \hrule height 0.3pt 
       \kern0.35ex
       \hbox{%
         \kern-0.1em
         \ensuremath{#1}%
         \kern-0.1em
       }%
     }%
   }%
}

\newcommand*\xxbar[1]{%
   \hbox{%
     \vbox{%
       \hrule height 0.3pt 
       \kern0.4ex
       \hbox{%
         \kern-0.1em
         \ensuremath{#1}%
         \kern-0.1em
       }%
     }%
   }%
}

\DeclareMathOperator{\sing}{Sing}

\makeindex

\begin{document}

\title{Holomorphic foliations tangent to Rolle-pfaffian hypersurfaces}

\author{Arturo Fern\'andez-P\'erez}
\address{Departamento de Matem\'atica ---  Universidade Federal de Minas Gerais}
\curraddr{Av. Ant\^onio Carlos 6627 --- 31270-901 --- Belo Horizonte, BRAZIL.}
\email{fernandez@ufmg.br}

\author{Rog\'erio   Mol}
\address{Departamento de Matem\'atica ---  Universidade Federal de Minas Gerais}
\curraddr{Av. Ant\^onio Carlos 6627 --- 31270-901 --- Belo Horizonte, BRAZIL.}
\email{rmol@ufmg.br}

\author{Rudy Rosas}
\address{Pontificia Universidad Cat\'olica del Per\'u}
\curraddr{Av. Universitaria 1801 --- Lima,
PERU.}
\email{rudy.rosas@pucp.edu.pe}


\thanks{The first and second authors are supported by CNPq Projeto Universal 408687/2023-1 ``Geometria das Equações Diferenciais Algébricas". The first author is supported by CNPQ PQ 306011/2023-9.  Third author is supported by
 Vicerrectorado de investigaci\'on de la Pontificia Universidad Cat\'olica del Per\'u}

\begin{abstract}
In this paper we study germs of holomorphic foliations at $(\C^{2},0)$ tangent to pfaffian hypersurfaces --- integral
hypersurfaces of real analytic $1$-forms --- satisfying the Rolle-Khovanskii condition. This hypothesis leads us to conclude that
such a foliation is defined by a closed meromorphic $1$-form, also allowing the classification of the simple models in its reduction
of singularities.
\end{abstract}

\maketitle

\section{Introduction}

A germ of singular holomorphic foliation at $(\C^{n},0)$, $n \geq 2$, is the object $\G$ defined by a germ of holomorphic $1$-form $\eta$,
with relatively prime coefficients in the local ring of holomorphic functions $\cl{O}_{n}$, satisfying the integrability
condition, i.e. $\eta\wedge \dd \eta = 0$. The geometric picture of this definition is the following: in a small neighborhood
of $0 \in \C^{n}$, where $\eta$ has a realization as a holomorphic $1$-form, outside the variety of
zeros of its coefficients --- its \emph{singular set} $\sing(\omega)$ ---  the distribution of complex hyperplanes defined by $\eta$ produces a decomposition of the space in integral hypersurfaces of $\eta$ --- the \emph{leaves} of the foliation --- which are immersed holomorphic manifolds of codimension one.
In this neighborhood, we say that  $\G$ is \emph{tangent to} a set $V$ --- or, equivalently, that
$V$ is \emph{invariant by} $\G$ --- if, whenever $p \in  V \setminus \sing(\F)$, the leaf of $\G$ containing $p$ lies entirely in $V$. In other words,
$V \setminus \sing(\G)$ is ``foliated'' by the leaves of $\G$. Questions concerning the existence of invariant sets is of significance, especially
when some structure of $V$ is required. For instance, if $V$ is an invariant  real analytic variety of real codimension one,
 we say that $V$ is a real analytic \emph{Levi-flat hypersurface}.

It is  natural to enquire the extent to which an invariant set $V$ determines  the properties of a foliation $\G$.
A remarkable result in this line, by D. Cerveau and A. Lins Neto \cite{cerveau2011}, says that if the holomorphic foliation  $\G$ is tangent
to  a real analytic Levi-flat
hypersurface then $\G$ has a meromorphic   \emph{first integral}, that is, its leaves are contained in the level sets of a meromorphic function (see its statement in Theorem \ref{teo-cerveau-linsneto} below).
Inasmuch as the analytic case is understood, we propose  as a next step  to weaken the analyticity hypothesis on $V$.
Without abandoning it completely, our idea is to work with sets that
keep some traits of analyticity. In this context,  Rolle-pfaffian hypersurfaces appear in our study.

Let us give some definitions. A \emph{Pfaffian hypersurface} is a maximal integral solution $V \subset W$, of real codimension one, of a real analytic
$1$-form $\omega$ defined  around the closure of a connected semi-analytic open  set  $W \subset \R^{n}$, for some $n \geq 2$.
This $1$-form is not necessarily integrable. If it is integrable, which will always happen when $V$ is \emph{transcendent} (i.e. not semi-analytic)  it then defines a real analytic singular
foliation  whose leaves are all pfaffian hypersurfaces.

Pfaffian hypersurfaces are immersed analytic manifolds and, thus, may  display  indocile topological behavior. In order to
tame their topology, a condition can be added, which is an analogue of the popular Rolle theorem of Calculus.
 It is the so-called \emph{Rolle-Khovanskii condition}  \cite{khovanskii1984,khovanskii1991}: a pfaffian hypersurface $V$ is \emph{Rolle} if every analytic curve in $W$ joining two points of $V$ is somewhere tangent to $\omega$. Or, reformulating this: if every analytic curve transverse to $\omega$ intersects $V$
at most once.  We say, in this case, that $V$ is \emph{Rolle-pfaffian}.
If the $1$-form $\omega$ is integrable, we say that
it defines a \emph{Rolle foliation} if all its integral hypersurfaces have the Rolle-Khovanskii property and  we refer to their leaves as \emph{Rolle leaves}.

If $V$ is a transcendent Rolle-pfaffian hypersurface, then it intercepts a trivializing neighborhood of the foliation defined by $\omega$ in a single plaque.
Thus, $V$ is real analytic at each of its points. However, analyticity fails in
the singular points  of $\omega$ contained in the adherence of $V$.
The hypothesis of being Rolle, however, imposes  at these points      conditions of non-oscillation.
Rolle-pfaffian hypersurfaces also enjoy  finiteness intersection properties characteristic of real analytic varieties (see the Finiteness
Theorem \ref{teo-finiteness} below \cite{moussu-roche1992}). Loosely speaking, Rolle-pfaffian hypersurfaces are objects that
``look'' analytic.

We can then state our guiding question: what can be said of a germ of singular holomorphic foliation $\G$
at $(\C^{n},0)$ tangent to a germ of Rolle-pfaffian hypersurface $V$?
We are here considering that the germ of pfaffian hypersurface $V$ is \emph{Rolle} if its  realization in some semi-analytic neighborhood of the origin of $\C^{n}$,
viewed as its underlying real space, has the Rolle-Khovanskii property.
Aiming at this question, we follow a script described below.
First, we circumvent the case where  $V$ is semi-analytic, since in this situation $\G$ is totally understood by Cerveau-Lins Neto's theorem
quoted above. Hence, we admit $V$ to be transcendent, which implies that its subjacent $1$-form $\omega$ is integrable and
defines a real analytic foliation $\F$ of real codimension one. The transcendency of $V$ also entails the fact that this
$\F$ and the holomorphic foliation $\G$ are tangent everywhere. Such an $\F$ is said to be a \emph{Levi-flat foliation}, for its leaves are, locally, real analytic
Levi-flat hypersurfaces. In this situation the so-called \emph{Levi} (or \emph{kernel}) \emph{foliation} of $\F$ is the holomorphic foliation $\G$. This
arrangement is well described  in the paper \cite{mol-fernandez-rosas2020}: 
in this case, $\G$ is defined by a closed meromorphic $1$-form, and this fact enables the determination of the possible
expressions for $\omega$ (see Theorem \ref{teo-Levi-closed} below and the ensuing paragraph).

So far the Rolle-Khovanskii condition has not been invoked. We do it   and investigate how it affects the
singularities of $\G$, limiting our study to dimension $n=2$. Even though restrictive, the two dimensional case somehow
gives information on foliations in arbitrary dimensions, for  a great deal of their geometry
is read in two-dimensional sections.
In dimension two, a holomorphic foliation has simple models, given by \emph{simple} or \emph{reduced} singularities.
An arbitrary foliation can be brought to these simple models through a \emph{reduction of singularities}, a proper holomorphic map given
by the composition of finitely many  quadratic blow-ups
(see Subsection \ref{subsection-singular-foliations}). A simple singularity, by definition, is one
induced by a holomorphic
vector field with non-nilpotent linear part, having eigenvalue ratio $\lambda \in \C \setminus \Q_{+}$, that we call \emph{multiplier} of the singularity.
The results
 in Section \ref{sec-simples} allow
the following possibilities for a germ of simple foliation $\G$ tangent to a Rolle-pfaffian hypersurface: either $\G$ is linearizable
with $\lambda \in \R \setminus \Q^{+}$ or, when $\G$ is non-linearizable, it is
either a resonant ($\lambda \in \Q_{-}$)  or a saddle-node ($\lambda = 0$) singularity.  In both cases  they are analytically normalizable, meaning
that  their formal normal forms, given in equations \eqref{eq-saddle-node} and \eqref{eq-resonant}, are actually analytic. In the case of a saddle-node, we also have that its \emph{weak multiplier} ($\mu$ in equation \eqref{eq-saddle-node})
is rational. In all possible alternatives,
$\G$ will be tangent to a Rolle foliation. These results are proved in  Propositions \ref{prop-hyperbolic}, \ref{prop-elliptic}, \ref{prop-saddle-node-rolle} and \ref{prop-resonant-rolle}.

Finally, in Section \ref{sec-consequences}, we consider an arbitrary germ of holomorphic foliation   $\G$ at $(\C^{2},0)$.
An invariant Rolle-pfaffian hypersurface, if lifted by its reduction of singularities, need not in general spread along the
exceptional divisor that replaces the singularity $0 \in \C^{2}$. There might be barriers such as simple singularities with
$\lambda \in \R_{+}$ --- that engender invariant Rolle-pfaffian hypersurfaces called \emph{nodal separators} --- saddle-node singularities
and also non-invariant (\emph{dicritical}) components in the exceptional divisor. Thus, we are compelled to work with
the rather stronger hypothesis of $\G$ being tangent to a Rolle foliation. Under this latter condition, we assemble in the statement of
Theorem \ref{teo-nondic-pfaffian} the following characterization of $\G$: it is defined by a closed meromorphic $1$-form and the simple models are those described in the preceding paragraph.
A more specific description of $\G$ is obtained in  Theorem  \ref{teo-logarithmic}, where it is additionally assumed
that $\G$ is \emph{non-dicritical} (i.e. it has an invariant exceptional divisor), of  \emph{generalized curve type} (i.e.
without saddle-nodes in its reduction of singularities), and that it has some simple singularity with non-rational $\lambda$.
We can conclude, under these hypotheses,  that $\G$ is \emph{logarithmic}, i.e. defined by a close meromorphic $1$-form with simple poles, having real
residues.
If  $f_{1},\ldots,f_{k} \in \cl{O}_{2}$ are the poles of this $1$-form and
 $\lambda_{1},\ldots,\lambda_{k} \in \R^{*}$ are its residues, then
 $f_{1}^{\lambda_{1}} \cdots f_{k}^{\lambda_{k}}$ is a multiform first integral for $\G$ and
 the invariant Rolle-pfaffian hypersurfaces are the level sets of the function
$ |f_{1}|^{\lambda_{1}} \cdots |f_{k}|^{\lambda_{k}}$.

\section{Pfaffian hypersurfaces and the Rolle-Khovanskii property}

Let $W \subset \R^{n}$ be a connected semi-analytic open set and $\omega$ be a non-singular  $1$-form in $W$,
defined and real analytic in
some neighborhood of the closure  $\ov{W}$ in $\R^{n}$.
We say that a triple  $\{V,\omega,W\}$  is a \emph{pfaffian
hypersurface}   if $V \subset W$ is an immersed manifold of dimension $n-1$ which
 is an  \emph{integral variety} of $\omega$, that is $T_{p}V = \ker \omega(p)$ for all $p \in V$, and is \emph{maximal} with respect to this property.
We will simply say that $V$ is a pfaffian hypersurface when no explicit reference to the neighborhood $W$ or to the $1$-form $\omega$ is needed.
Actually, in this work we shall consider germs of real analytic $1$-forms at $0 \in \R^{n}$ whose singular
set is of real codimension at least two.
In this case,
we will suppose these germs  to be realized in  some
neighborhood of the compact $\ov{U}$, where $U$ is a small semi-analytic neighborhood   of $0 \in \R^{n}$ (for instance, $U$ can be taken to be a ball of small radius centered at $0 \in \R^{n}$). We then take
the semi-analytic open set $W = U \setminus \sing(\omega)$ and consider pfaffian
hypersurfaces of the form $\{V,\omega,W\}$. In this way, the notion of \emph{germ of pfaffian hypersurface} at $0 \in \R^{n}$ is well defined.

 We say that a pfaffian hypersurface $\{V,\omega,W\}$ is of \emph{Rolle} type if it satisfies  the Rolle-Khovanskii property:
 every analytic curve intersecting $V$ twice transversely must have an intermediate point of tangency with $\omega$.
 In other words, if
 $\gamma:[0,1] \to W$ is an analytic map  such that  $\gamma(0), \gamma(1) \in V$
then there exists $t_{0} \in [0,1]$ such that $\gamma'(t_{0}) \in \ker \{\omega(\gamma(t_{0}))\}$.
 This can be rephrased in the following manner: every analytic curve on $W$ transverse to $\omega$ intercepts $V$ at most once.
 In this case, we will say that $\{V,\omega,W\}$ --- or simply $V$, if the context makes clear who $\omega$ and $W$ are ---
 is a Rolle-pfaffian hypersurface.

 We have that the pfaffian hypersurface $\{V,\omega,W\}$ is  Rolle  in the following two cases:
 \begin{enumerate}[label=(\roman*)]
 \item if $V$ is a separating hypersurface in $W$ \cite{khovanskii1984};
 \item if $W$ is simply connected \cite{moussu-roche1992}.
 \end{enumerate}
 We say that $V$ is \emph{separating} if $W\setminus V$ has two connected components and $V$ is their common boundary in $W$.
For instance, if $V \subset W$ is the connected non-singular   level set of a  real-analytic function $F$ defined and non singular in $W$, then $V$ is pfaffian for the
 $1$-form $dF$. Besides, it is separating ---
 if $V = \{F=c\}$, then it separates the connected components
 $\{F<c\}$ and $\{F>c\}$ of $W$. Thus, by (i) above,  $V$ is Rolle for $\dd F$ or for any other real analytic $1$-form $\omega$ such that $\{V,\omega,W\}$ is a pfaffian
 hypersurface.
One such Rolle-pfaffian hypersurface will be called \emph{semi-analytic}.
A pfaffian hypersurface $\{V,\omega,M\}$ which is not semi-analytic will be called \emph{transcendent}.

Rolle-pfaffian hypersurfaces have their roots in the works of A. G. Khovanskii  \cite{khovanskii1984,khovanskii1991}.
They have many nice properties of analytic hypersurfaces,   particularly  some finiteness properties,
as described in the following theorem by R. Moussu and C. Roche
\cite{moussu-roche1992}:
 \begin{teo}[Finiteness Theorem]
 \label{teo-finiteness}
 If $X$ is a relatively compact semi-analytic set in $\R^{n}$ and $\{V_{i},\omega_{i},W\}$,
 $i=1,\ldots,k$, are Rolle-pfaffian  hypersurfaces,
 then the number of connected components of $X \cap V_{1} \cap \cdots \cap V_{k}$ is finite and bounded by a number  which depends only on $W$, $X$ and the
 pfaffian forms $\omega_{i}$.
 \end{teo}

The Rolle property for pfaffian hypersurface germifies: we say that a germ of pfaffian hypersurface $V$ at $0 \in \R^{n}$, with defining $1$-form $\omega$,
 is Rolle if there exists a semi-analytic neighborhood   of $0 \in \R^{n}$, with $\omega$  realized in
a neighborhood of  $\ov{U}$, such that, setting $W = U \setminus \sing(\omega)$, the triple $\{V,\omega,W\}$ has the Rolle-Khovanskii property.
It is easy to see that if this property holds for $U$ it will also hold for any semi-analytic neighborhood $U'$ of $0 \in \R^{n}$ such
that $U' \subset U$.

Let $\omega_{1},\ldots,\omega_{\ell}$ be $1$-forms defined in an open set  $W \subset \R^{n}$. We say that an immersed manifold
$L \subset W$ is \emph{invariant} by the pfaffian system $\{\omega_{1},\ldots,\omega_{\ell}\}$ if, for every $p \in W$,
$T_{p} L \subset \ker \omega_{1}(p) \cap \cdots \cap \ker \omega_{\ell}(p)$.
We have the following elementary fact:
\begin{prop}
 \label{prop-integrable}
  Let $W \subset \R^{n}$ be an  open set and $\omega_{1},\ldots,\omega_{\ell}$, where $\ell < n$, be  real-analytic $1$-forms on $W$. Let $L$ be
an $(n-\ell)$-dimensional immersed submanifold of $W$ invariant by the pfaffian system $\{\omega_{1},\ldots,\omega_{\ell}\}$.
Then, for any fixed $j$, with $1\leq j \leq \ell$, the $(\ell +2)$-form
\[ \dd \omega_{j} \wedge \omega_{1} \wedge \cdots \wedge  \omega_{\ell} \]
vanishes over $L$.
 \end{prop}
\begin{proof}
The problem is local. Fix $p \in L$ and take a small neighborhood $U$ of $p$ and   analytic coordinates $(x_{1},\ldots,x_{n}) \in \R^{n}$
such that $p$ is the origin and  the connected component of $L \cap U$ containing $p$ has equations $x_{1} = \ldots = x_{\ell} = 0$.
The invariance condition of $L$    gives that, for $j=1,\ldots,\ell$, we can write, in a small neighborhood of
$0 \in \R^{n}$,
\[ \omega_{j} = \sum_{i=1}^{\ell} A^{j}_{i} \dd x_{i} +  x_{i} \eta^{j}_{i} ,\]
where the $A^{j}_{i}$ are functions and the
$\eta^{j}_{i}$ are $1$-forms involving only $\dd x_{\ell+1}, \ldots, \dd x_{n}$, real analytic near the origin.
Thus, for $j=1,\ldots,\ell$,
\begin{equation}
\label{eq-domega}
  \dd \omega_{j} = \sum_{i=1}^{\ell}\left( \dd A^{j}_{i} - \eta^{j}_{i} \right) \wedge \dd x_{i}    +  \sum_{i=1}^{\ell} x_{i}  \dd \eta^{j}_{i}.
\end{equation}
On the other hand
\begin{equation}
\label{eq-wedgeomega}
\omega_{1} \wedge \cdots \wedge  \omega_{\ell} = \Phi \dd x_{1} \wedge \cdots \wedge   \dd x_{\ell} + \sum_{i=1}^{\ell} x_{i} \theta_{i} ,
\end{equation}
where $\Phi$ is a function ---  the determinant of the  $\ell \times \ell$ matrix $(A^{j}_{i})$ --- and the $\theta_{i}$'s are $\ell$-forms.
It is apparent that the wedge of the first sum to the right of \eqref{eq-domega} with the first term to the right of  \eqref{eq-wedgeomega} vanishes, and we find
\[ \dd \omega_{j} \wedge \omega_{1} \wedge \cdots \wedge \omega_{\ell} =  \sum_{i=1}^{\ell} x_{i} \zeta_{i}, \]
where the $\zeta_{i}$'s are real analytic  $(\ell+2)$-forms. This proves the proposition.
\end{proof}

Proposition \ref{prop-integrable} has the following consequence:

\begin{cor}
\label{cor-transcendent}
If the pfaffian hypersurface $\{V,\omega,W\}$ is transcendent, then $\dd \omega \wedge \omega \equiv 0$ on $W$.
\end{cor}
\begin{proof} It suffices to apply the proposition for
$\ell = 1$ and $\omega = \omega_{1}$. We have $\omega \wedge \dd \omega = 0$ over $V$. Thus, the coefficients  of
 $\dd \omega \wedge \omega$ are   real-analytic functions vanishing over  $V$, a transcendent set of real codimension one on $W$. Hence, all
 these coefficients vanish everywhere, which proves the corollary.
\end{proof}
Let $\omega$ be a germ of real analytic $1$-form at $0 \in \C^{n}$, whose coefficients have no non-trivial common factor in $\A_{n \R}$, the ring of real analytic functions with real values at $(\C^{n},0)$. If
 $\omega \wedge \dd \omega =0$
we say   that $\omega$ defines a germ \emph{singular real analytic foliation} $\F$
whose \emph{singular set} is $\sing(\F) := \sing(\omega)$. A foliation is defined up to multiplication by a unity in $\A_{n \R}$.
In a small neighborhood of the origin where
the germ $\omega$ is realized, it defines
a real analytic foliation of real codimension one outside $\sing(\omega)$.
It follows  from   Corollary \ref{cor-transcendent}   that a germ of transcendent pfaffian hypersurface $V$ defines a germ of real analytic foliation $\F$ of codimension one,
by eventually cancelling the common factors of the germ of its defining $1$-form.
Observe that two real analytic $1$-forms $\omega$ and $\omega'$, both having coefficients without non-trivial common factors,  defining,
 in a neighborhood of $0 \in \C^{n}$, the same distribution of hyperplanes
outside $\sing(\omega) \cup \sing(\omega')$, differ one from the other by the multiplication of a unity in $\A_{n \R}$.
This remark   gives as a consequence the fact  that the foliation $\F$ defined by the  transcendent pfaffian hypersurface $V$ is intrinsic.
Indeed, it follows from transcendency that two germs of real analytic $1$-forms $\omega$ and $\omega'$, having coefficients without non-trivial common factors, if tangent to $V$ satisfy $\omega \wedge \omega' \equiv 0$ and, hence, define the same germ of real analytic foliation.
If all leaves of the foliation $\F$
satisfy the Rolle-Khovanskii property, we
will say that $\F$ is a \emph{Rolle foliation} and that its leaves are
\emph{Rolle leaves}.

We finish this section by observing that if the pfaffian hypersurface $\{V,\omega,W\}$ is  Rolle then $V$ is closed in $W$. This is evident if $V$ is semi-analytic.
On the other hand, if $V$ is transcendent, $\omega$ defines a real-analytic foliation  $\F$. If $p \in W \setminus V$ is an accumulation point
of $V$, then,  in a trivialization of $\F$ around $p$,  any  analytic curve transverse to $\F$ at $p$ will meet $V$ infinitely many times,
saying that $V$ is not Rolle. In particular, all leaves of a Rolle foliation are closed.

\section{Holomorphic foliations and Levi-flat hypersurfaces}

\subsection{Singular holomorphic foliations}
\label{subsection-singular-foliations}
We recall here some basic concepts of the theory of holomorphic foliations.
They can be found in several  texts, for instance \cite{canocerveau2013}.
A germ of singular holomorphic foliation of codimension one at $(\C^{n},0)$ is the object
defined, algebraically, by a germ of holomorphic $1$-form $\eta$ at $0 \in \C^{n}$, having a singular set of codimension at least two
and satisfying the integrability condition, i.e. $\eta \wedge \dd \eta = 0$, all this up to multiplication
by a non-unity function in $\cl{O}_{n}$.
Geometrically, this corresponds to the attribution, in a small neighborhood of
$0 \in \C^{n}$ where $\eta$ is realized, of an  analytic set corresponding to singular set $\sing(\eta)$ and of
a structure of \emph{leaves} of a regular holomorphic foliation, corresponding to  the   integral hypersurfaces  of $\eta$.
If $\G$ denotes the singular holomorphic foliation given by $\eta$, its \emph{singular set} is defined as $\sing(\G) := \sing(\eta)$.
Hereafter, as a convention, we will omit the word ``singular'' in the expression  ``singular holomorphic foliation''.

A germ of closed meromorphic $1$-form $\tau$ defines a holomorphic foliation $\G$, by taking $\eta= \varphi \tau$, where $\varphi$ is a germ of meromorphic function
at $(\C^{n},0)$
that cancels the components of zeros and poles of $\tau$.
If the polar set of $\tau$ has irreducible
components with equations $f_{1},\ldots,f_{k} \in \cl{O}_{n}$, then it can be written as
\begin{equation}
\label{eq-closed-form}
\tau = \sum_{i=1}^{k} \lambda_{i} \frac{\dd f_{i}}{f_{i}} + \dd \left( \frac{h}{f_{1}^{m_{1}-1} \cdots f_k^{m_{k}-1} } \right),
\end{equation}
where $ h \in \cl{O}_{n}$ and, for each $i=1,\ldots, k$, we have that $\lambda_{i} \in \C$ and $m_{i}$ is the order of the pole defined by $f_{i}$
\cite{cerveau1982}.
A particular case is when $\G$ has a \emph{meromorphic first integral}, which means that the leaves of $\G$ are contained in
the level sets of a meromorphic function $H$ (we include here the possibility of $H$ being holomorphic). In this case, $\tau = \dd H$ is a closed meromorphic $1$-form defining $\G$.
The following simple remark will be used later: if the holomorphic foliation $\G$ is defined by the closed holomorphic $1$-form $\tau$
and does not admit a meromorphic first integral, then any other closed meromorphic $1$-form defining $\G$ is written as
$\alpha \tau$, for some $\alpha \in \C^{*}$.

Our main results concern  holomorphic foliations  in dimension two, in which case the integrability condition always applies.
When $n=2$, there are ``simple'' models for holomorphic foliations:
 $\G$ has a \emph{simple singularity}  if
it is induced by a vector field whose linear part is non-nilpotent and,
when its eigenvalues $\lambda_{1}, \lambda_{2}$ are both  non-zero, they satisfy $\lambda = \lambda_{1} / \lambda_{2} \not\in \Q_{+}$. This latter type of singularity is said to be \emph{non-degenerate}. When exactly one eigenvalue is zero, the singularity
is called \emph{saddle-node}.
Non-degenerate singularities are classified as \emph{hyperbolic}, \emph{elliptic} or \emph{resonant},
corresponding to the cases $\lambda \in \C \setminus \R$, $\lambda \in \R \setminus \Q$ or $\lambda \in \Q_{-}$. We will refer to
this $\lambda$ as the \emph{multiplier} of the singularity.
A germ of holomorphic foliation with a non-degenerate simple singularity is \emph{linearizable} if
 there are local holomorphic coordinates $(x,y)$ at $(\C^{2},0)$ in which it is induced by a $1$-form of the
type $\eta = x \dd y - \lambda y  \dd x$.

A \emph{separatrix} is an irreducible formal curve invariant by the holomorphic foliation. A simple singularity
has exactly two separatrices, both smooth and transverse to each other, corresponding to the eigenspaces of two eigenvalues.
In the non-degenerate case, both separatrices are analytic. In the saddle-node case, the one corresponding
to the non-zero eigenvalue is analytic --- the \emph{strong separatrix} --- whereas the one associated with
the zero eigenvalue can be a purely formal curve --- the \emph{weak separatrix}.

In dimension two, there exists a \emph{reduction of singularities} for a holomorphic foliation $\G$  \cite{seidenberg1968,camacho1982}.
This means that there exists
a proper holomorphic map
$\pi: (\tilde{M},\D) \to (\C^{2},0)$ --- a composition of quadratic blow-ups, where $\D = \pi^{-1}(0)$
is a divisor consisting of a finite union projective lines with normal crossings and $\tilde{M}$ is a germ of surface
around $\D$ --- which
 transforms $\G$ into a foliation
$\tilde{\G} =\pi^{*} \G$  defined in a neighborhood of $\D$ on $\tilde{M}$, having only
simple singularities. Besides,  $\pi$  desingularizes
 the set of separatrices, meaning that the separatrices of $\G$ are taken by $\pi$ into disjoint smooth curves, all of them transverse
 to the exceptional divisor $\D$.
A reduction of singularities  is not unique, but we can fix a minimal
reduction of singularities, which is unique, up to  isomorphism.

A holomorphic foliation $\G$, in dimension $n=2$,  is \emph{non-dicritical} if, given a reduction of singularities as above, all irreducible components of $\D$ are invariant by $\tilde{\G}$. Otherwise,   it is said to be \emph{dicritical}. A non-invariant component of $\D$ of a dicritical foliation is fibered by leaves of $\tilde{\G}$ which, projected
 by $\pi$, give rise to infinitely many separatrices of $\G$. On the other hand, a non-dicritical foliation
 has a finite number of separatrices. They correspond to separatrices, not contained
 in the exceptional divisor $\D$, of the simple singularities of $\tilde{\G}$.

\subsection{Levi-flat foliations}

The main references for this subsection are \cite{cerveau2004} and \cite{mol-fernandez-rosas2020}.
Let $U \subset \C^{n}$ be an open subset and
 $H \subset U$ be a real analytic submanifold of real codimension one.
For each   $p \in H$, there is a unique complex vector space of dimension $n-1$ contained in the tangent space $T_{p}H$, which is  given by $\cl{L}_{p} = T_{p}H \cap i T_{p}H$, where $i = \sqrt{-1}$.
When the real analytic distribution of complex hyperplanes $p \in H \mapsto \cl{L}_{p}$  is integrable, in the sense of Frobenius, we say that $H$ is a {\em Levi-flat} hypersurface.
Thus, $H$ is foliated by   immersed complex manifolds of complex dimension $n-1$, defining the so-called \emph{Levi foliation}.
This foliation, in principle, does not extend to a holomorphic foliation in a neighborhood of $0 \in \C^{n}$. If this happens,
this extension is unique and
the foliation   is characterized by the following result:
\begin{teo}[Cerveau-Lins Neto,\cite{cerveau2011}]
\label{teo-cerveau-linsneto}
Let $\G$ be a germ  of holomorphic foliation of   codimension one at $(\C^{n},0)$
 tangent to a germ of
 real analytic Levi-flat hypersurface. Then $\G$ admits a meromorphic first integral.
\end{teo}

If the leaves of a real analytic foliation  $\F$ defined by an integrable real analytic $1$-form $\omega$, in a neighborhood $U$ of $0 \in \C^{n}$, are, locally,  Levi-flat hypersurfaces, we say that  $\F$ ---
or   $\omega$ --- is   \emph{Levi-flat}. In this case, there exists a real analytic   foliation $\LL = \LL(\F)$ on $U$ whose leaves are immersed complex manifolds of dimension $n-1$, entirely contained in the leaves of $\F$. We also refer to this foliation as \emph{Levi foliation}, or, to avoid confusion, as \emph{kernel foliation}.
It is worth emphasizing that, although the leaves of $\LL$   are complex manifolds,
 as a foliation, it has in principle real analytic regularity.

The notions above   can be germified, making sense the expression \emph{germ of Levi-flat foliation}.
A germ of  real analytic $1$-form $\omega$ at $0 \in \C^{n}$, defining a real analytic foliation $\F$, can be written as
\begin{equation}
\label{omega-def}
\omega = \sum_{j=1}^{n} A_{j} \dd x_{j} + B_{j} \dd y_{j},
\end{equation}
where $A_{j}, B_{j} \in \A_{n \R}$, for $j = 1, \ldots, n$, are
without non-trivial common factors.
We have that $\sing(\F)= \sing(\omega)$  is a
real analytic set of real codimension at least two, defined by the vanishing of all coefficients of $\omega$.
In complex coordinates $z = (z_{1},\cdots,z_{n})$  in $\C^{n}$, where
$z_{j} = x_{j} + iy_{j}$,
 we write
 \begin{equation}
\label{eq-omega}
\omega     =    \sum_{j=1}^{n} \frac{A_{j} - i B_{j}}{2} \dd z_{j} +
\sum_{j=1}^{n} \frac{A_{j} + i B_{j}}{2} \dd \bar{z}_{j}
  =   \frac{\eta + \bar{\eta}}{2} =  \re(\eta
 ) .
\end{equation}
The $1-$form
$\eta    =  \sum_{j=1}^{n} (A_{j} - i B_{j}) \dd z_{j}$
defines, outside $\sing(\F)$,  the intrinsic  distribution of complex hyperplanes associated with $\omega$.
Setting
\begin{equation}
\label{eq-omega-sharp}
 \omega^{\sharp}    = \im (\eta )  = \frac{\eta - \bar{\eta}}{2 i}
  =    \sum_{j=1}^{n} -B_{j}  \dd x_{j} + A_{j} \dd y_{j},
 \end{equation}
we have that  the germ of real analytic foliation $\F$ defined by $\omega$ is Levi flat if and only if
 the pfaffian system
$\{\omega, \omega^{\sharp}\}$ is integrable, which in turn is
  equivalent to
\begin{equation}
\label{eq-integrability}
 \dd \omega \wedge \omega \wedge \omega^{\sharp} = 0 \qquad \text{and} \qquad \dd \omega^{\sharp} \wedge \omega \wedge \omega^{\sharp} = 0 .
\end{equation}
In this case, the pfaffian system $\{\omega, \omega^{\sharp}\}$ defines the kernel foliation $\LL$.
As mentioned earlier,  $\LL$ in principle has real analytic regularity.
Holomorphic foliations that are kernel foliations  of   Levi-flat foliations  have the following characterization:
\begin{teo} {\rm \cite[Th.2]{mol-fernandez-rosas2020}}
\label{teo-Levi-closed}
A  germ of holomorphic foliation at $(\C^{n},0)$ is the kernel foliation of a germ of real-analytic  Levi-flat foliation if and only if it is defined by a closed meromorphic  $1$-form.
\end{teo}
In the theorem's context, suppose that $\G$ is a holomorphic foliation
 tangent to a Levi-flat foliation $\F$,
with $\G$   defined by a meromorphic $1$-form $\tau$. If    $\G$ does not admit a meromorphic first integral,   then
$\F$ is defined by a real meromorphic $1$-form $\re(\tau) = (\tau + \bar{\tau})/2$, where $\tau$ is a closed meromorphic $1$-form defining $\G$
 \cite[Th.2, Ex.5.1]{mol-fernandez-rosas2020}.
Thus,  all Levi-flat foliations tangent to $\G$   are defined by real meromorphic $1$-forms
of the kind $\re(\alpha \tau) = (\alpha \tau + \bar{\alpha }\bar{\tau})/2$, where $\alpha \in \C^{*}$.


\subsection{Pfaffian hypersurfaces tangent to holomorphic foliations}
Let $V$ be a germ of pfaffian hypersurface defined at $0 \in \C^{n}$ by a germ of real-analytic $1$-form $\omega$.
By that we mean  that there is a semi-analytic neighborhood $U$ of the origin, with  $\omega$ and $V$ being realized in a neighborhood of $\bar{U}$,
 such that $\{V,\omega,W\}$ is pfaffian, where $W = U \setminus \sing(\omega)$.
In particular, $0 \in \ov{V}$, since we are evidently supposing that the germ of $V$ is non-trivial.
We say that $\G$  is \emph{tangent to}  $V$ or, equivalently,
$V$ is \emph{invariant} by $\G$ if $T_{p}\G \subset T_{p}V$
for every $p \in V$. In other words, in $W$, $V$ is a Levi-flat hypersurface  whose Levi foliation
is the restriction of $\G$ to $V$.
If $V$ is semi-analytic in $U$, then it is contained in a (possibly singular) real analytic hypersurface, which will also be
Levi-flat. In this case, we can apply Cerveau-Lins Neto's result (Theorem \ref{teo-cerveau-linsneto}) in order to conclude that
$\G$ has a meromorphic first integral.

Hence, let us suppose that the invariant  pfaffian hypersurface is  transcendent. In this case, by Corollary \ref{cor-transcendent}, $\omega$ is integrable, defining
a germ of singular real analytic foliation of codimension one in $U$. More than that, in this case
$\omega$ defines a Levi-flat foliation, as shown next:
\begin{prop}
\label{prop-levi-flat-foliation}
Let $V$ be a germ of  transcendent pfaffian hypersurface at $0 \in \C^{n}$ which is a solution of a germ of real analytic $1$-form
$\omega$. Suppose that $V$ is tangent to a germ of singular holomorphic foliation $\G$.
Then $\omega$ defines a singular real analytic Levi-flat foliation whose kernel foliation is $\G$.
\end{prop}
\begin{proof} As noticed above, $\omega$ is integrable as a consequence of Corollary \ref{cor-transcendent}, defining a real analytic foliation of real codimension one.
In order to see that this foliation is Levi-flat, we start by fixing a leaf $L$ of $\G$ contained in $V$. We  invoke   Proposition $\ref{prop-integrable}$ for $L$, setting $\ell = 2$,  $\omega_{1} = \omega$
and $\omega_{2} = \omega^{\sharp}$, and then conclude that the $4$-forms in \eqref{eq-integrability} vanish over $L$. Then, varying $L$ over $V$, we have
that they actually vanish over $V$. Now, once again, considering that
 $V$ is transcendent of codimension one, we conclude that  both equations in \eqref{eq-integrability} actually vanish in
a neighborhood of $0 \in \C^{n}$,     proving that $\omega$ defines a   Levi-flat foliation $\F$.
In order to see that the kernel foliation of $\F$ is $\G$, it is enough to see that the tangency locus of $\F$ and $\G$ is
a real analytic set that contains $V$. Since $V$ is transcendent, this locus is a full neighborhood of $0 \in \C^{n}$, proving the result
\end{proof}


\section{Simple singularities}
\label{sec-simples}

The results in this section concern germs of simple singularities of holomorphic foliations at $(\C^2,0)$  tangent to
germs Rolle-pfaffian hypersurfaces. We can thus consider that both, the hypersurface $V$
and the holomorphic foliation $\G$, are germs at $(\C^2,0)$,  realized   in a neighborhood
of  $\ov{U}$, where $U=\Delta$ is the   unit polydisc centered at the origin.
It is implicit in this case that $0 \in \C^{2}$
belongs to the closure of $V$.

\subsection{Hyperbolic singularities}

We first consider hyperbolic singularities, those corresponding to a multiplier
 $\lambda     \in \C \setminus \R$. In this case, the foliation is linearizable, i.e.
there are  holomorphic coordinates $(x,y)$ in which the holomorphic foliation $\G$
is defined by the $1$-form $\eta = y \dd x - \lambda x \dd y$.
Singularities of this kind are incompatible with the existence of invariant  Rolle-pfaffian hypersurfaces:

\begin{prop}
\label{prop-hyperbolic}
Let $\G$ be a germ of holomorphic foliation at $0 \in \C^{2}$ having a  singularity of hyperbolic type. Then there are no   Rolle-pfaffian hypersurfaces  invariant by $\G$.
\end{prop}
\begin{proof}
We work by contradiction and admit that $\G$ is tangent to a Rolle-pfaffian hypersurface $V$.
It must be transcendent, for otherwise Theorem \ref{teo-cerveau-linsneto}    would imply the existence of a holomorphic first integral, which
would result in $\lambda \in \Q_{-}$.
Hence, by Proposition \ref{prop-levi-flat-foliation}, its subjacent $1$-form
$\omega$ defines a real analytic Levi-flat foliation $\F$ whose Levi foliation is $\G$.
We take holomorphic coordinates $(x,y)$, defined
in a neighborhood of $\ov{\Delta}$, in which $\G$
is defined by the $1$-form $\eta = y \dd x - \lambda x \dd y$ or, equivalently, by
the meromorphic $1$-form $\tau = \frac{\dd x}{x} - \lambda \frac{\dd y}{y}$.
By \cite[Ex.5.1]{mol-fernandez-rosas2020}, the Levi-flat foliation $\F$ is defined by the real meromorphic $1$-form
\[ \re(\alpha \tau) = \frac{\alpha \tau  + \bar{\alpha} \bar{\tau}}{2}
= \alpha \frac{\dd x}{x} + \bar{\alpha}\frac{\dd \bar{x}}{\bar{x}} -  \alpha\lambda \frac{\dd y}{y}  - \bar{\alpha} \bar{\lambda} \frac{\dd \bar{y}}{\bar{y}}, \]
for some $\alpha \in \C^{*}$. Since $\lambda \in \C \setminus \R$, we have that either $\alpha \in \C \setminus \R$ or
$\alpha \lambda \in \C \setminus \R$.
Assume that the first case happens. Take the transverse section $\Sigma = \{ (x,1); |x|<1\}$. The real one-dimensional foliation
$\F|_{\Sigma}$ is defined by $\alpha \bar{x} \dd x + \bar{\alpha} x \dd \bar{x}$. Since $\alpha \in \C \setminus \R$, this foliation
has a focus dynamics around $x=0$. Thus any leaf will have infinitely many points of intersection with a line in $\Sigma$ crossing the origin, and the same is true
for $V \cap \Sigma$, which, by definition, says that $V$ is not Rolle. The same argument works for $\alpha \lambda \in \C \setminus \R$, by taking a section transverse to the $x$-axis.
\end{proof}

\subsection{Elliptic singularities}

The case of elliptic singularities,   those with multiplier
  $\lambda \in \R \setminus \Q$, contrasts  with what Proposition \ref{prop-hyperbolic} describes for hyperbolic singularities.
They may be tangent
to Rolle-pfaffian hypersurfaces, as  the following example describes:
\begin{example}
\label{ex-nodal}
{\rm If $\lambda \in \R_{+}$, we say that $\G$ has a
\emph{nodal singularity}. In this case
$\G$ is linearizable in analytic coordinates $(x,y)$, in which $\G$ is defined by the linear $1$-form $\eta = ydx - \lambda x \dd y$.
Then, $\G$ is tangent to the hypersurfaces $V_{c}$ of equation  $|x|= c|y|^{\lambda}$, where $c>0$.
These are pfaffian hypersurfaces, called \emph{nodal separators} in the terminology of \cite{marin2012}, tangent to the real
analytic $1$-form
\[\omega = |x|^{2} |y|^{2} \re \left(  \frac{ \dd x}{x} - \lambda \frac{\dd y}{y} \right) = \frac{1}{2}
(\bar{x} |y|^{2}  \dd x + x|y|^{2} \dd \bar{x} + \lambda|x|^{2}\bar{y} \dd y +\lambda|x|^{2}y \dd \bar{y} ), \]
which is regular outside the coordinate axes.
These hypersurfaces are Rolle, since they are separating, being the levels of the
 the function $f(x,y) = |x||y|^{-\lambda}$, which is real analytic for $x,y \neq 0$. Remark that, for a given $c>0$, the only accumulation
 point of $V_{c}$ on the coordinate axes is the  origin.
 }\end{example}

A foliation with an elliptic singularity with multiplier    $\lambda \in \R_{-}$ may not be linearizable.
In the linearizable case, when   $\G$ is defined by the linear $1$-form $\eta = ydx - \lambda x \dd y$ as above,
the hypersurfaces $ |x|= c|y|^{\lambda}$, with $c>0$, are Rolle-pfaffian and invariant by $\G$, but   do
not accumulate to $0 \in \C^{2}$. However,  there are hypersurfaces in this family arbitrarily near to the origin.
On the other hand, non-linearizable elliptic foliation are not tangent to  pfaffian hypersurfaces:

\begin{prop}
\label{prop-elliptic}
Let $\G$ be a germ of holomorphic foliation at $0 \in \C^{2}$ having a non-linearizable elliptic singularity.
 Then   $\G$ is not tangent to a   pfaffian hypersurface.
\end{prop}
\begin{proof}
Suppose, aiming at a contradiction, that $\G$ is tangent to a Rolle-pfaffian hypersurface $V$ accumulating to $0 \in \C^{2}$.
As argued in Proposition \ref{prop-hyperbolic},  $V$ must be transcendent, so that
its subjacent real-analytic $1$-form defines a Levi-flat foliation whose Levi foliation is $\G$.
By \cite[Lem. 2]{belko2003}, the holonomy maps corresponding each one of the  separatrices of $\G$ are linearizable.
This is a contradiction, since, in this case,
$\G$ itself would be linearizable  \cite[Th. 2]{mattei1980}.
\end{proof}

\subsection{Saddle-node singularities}
\label{subsection-saddlenode}

Let $\G$ be a germ of singular holomorphic foliation having a simple singularity at $0 \in \C^{2}$ of
saddle-node type, i.e. engendered by a holomorphic vector field whose linear part has exactly one zero eigenvalue.
Then there are formal coordinates $(x,y)$ in which $\G$ is defined by the $1$-form
\begin{equation}
\label{eq-saddle-node}
\eta = x(1+\mu y^k)  \dd y - y^{k+1} \dd x,
\end{equation}
where $k \geq 1$ and $\mu \in \C$ \cite[p.64]{martinet1982}.
The saddle-node singularity
is \emph{analytically normalizable} if these coordinates $(x,y)$   can  be taken to be analytic.
The $1$-form $\eta$  has the meromorphic integrating factor
$F =  1/(x y^{k+1})$, i.e.
\begin{equation}
\label{eq-saddle-node-closed}
\tau =   \frac{\eta}{x y^{k+1}} =    \left( \frac{\mu }{y} + \frac{1}{y^{k+1}} \right)  \dd y - \frac{\dd x}{x}
\end{equation}
is closed, giving rise to  the multivalued function $xy^{-\mu}e^{\frac{1}{k}y^{-k}}$
as a first integral.
The strong separatrix corresponds to  the $x$-axis and the weak separatrix to the $y$-axis.  The numbers $k$ and $\mu$  are formal invariants. The integer $k+1$ is the Milnor number of the saddle-node (see \cite{camacho1984}) and
can be seen as a measure of the extent to which
the singularity fails to be non-degenerate.
In the lines below,
the complex number $\mu$ will be referred to  as the \emph{weak multiplier}.
This number coincides with the \emph{Camacho-Sad index} (see \cite{camacho1982}) of $\G$ with respect to the weak separatrix.

One relevant dynamical property of a saddle-node singularity is  that, in a small neighborhood of the origin, all the leaves outside
the two separatrices
accumulate on the  strong separatrix. This can be seen from the fact that the holonomy  of the latter, with respect to a one-dimensional  section $(\Sigma,p) \simeq (\C,0)$ transverse to the strong separatrix, written in a holomorphic coordinate $z \in \Sigma$,
has the form $\Phi(z) = z + \lambda z^{k+1} + ...$, for some
$\lambda \in \C^{*}$, $h.o.t.$ standing for higher order terms
  (\cite[p.482]{mattei1980}, \cite{canocerveau2013}). The dynamics of this map has the so-called \emph{flower structure}, as a consequence of
Leau-Fatou Flower Theorem \cite[Th. 2.12]{bracci2010}. It says in particular that  all   orbits of $\Phi$ outside $z=0$ accumulate, in positive or negative time, to $z=0$.
Hence, a saddle-node singularity is incompatible with the existence of a meromorphic first integral, since, in this case, all leaves in a neighborhood of the origin would be closed, resulting
in a holonomy map  with finite orbits. Hence, Theorem \ref{teo-cerveau-linsneto} rules out
the existence of a real analytic Levi-flat hypersurface at a saddle-node singularity.

If $\G$ is an analytically normalizable saddle-node singularity, then  the weak separatrix is convergent. In this case, its the holonomy
with respect to
a transverse section $(\Sigma,p) \simeq (\C,0)$, with an analytic coordinate $z$,  has the form
$\Phi(z) = e^{2 \pi i \mu} z + h.o.t.$. This is a straightforward calculation that can be seen,   for instance,
in \cite[Sec. 6.6.3]{canocerveau2013}.

Assume now that the saddle-node singularity is tangent to a
germ of   Rolle-pfaffian hypersurface $V$.
The non-existence  of a meromorphic first integrals implies that $V$
must be transcendent.
Thus, by Proposition \ref{prop-levi-flat-foliation}, $\G$ is tangent to a Levi-flat foliation and, by \cite[Lem. 1]{belko2003}, the saddle-node singularity
is analytically normalizable.
Taking normalizing coordinates
 $(x,y)$ as above,   the closed meromorphic $1$-form $\tau$ in \eqref{eq-saddle-node-closed} defines $\G$.
In these coordinates, any other closed meromorphic $1$-form defining $\G$ is of the form
$\alpha \tau$, for some $\alpha \in \C^{*}$.
By \cite[Ex.5.1]{mol-fernandez-rosas2020}, for each $\alpha \in \C^{*}$, the real meromorphic $1$-form
\begin{equation}
\label{eq-meromorphic-levi}
 \re(\alpha \tau) =  \frac{ \alpha \tau + \bar{\alpha} \bar{\tau}}{2}
 = \frac{1}{2} \left(  \alpha  \left( \frac{\mu }{y} + \frac{1}{y^{k+1}} \right)  \dd y
 + \bar{\alpha} \left( \frac{\bar{\mu} }{\bar{y}} + \frac{1}{\bar{y}^{k+1}} \right)  \dd \bar{y}
  -  \alpha \frac{\dd x}{x}
   - \bar{\alpha} \frac{\dd \bar{x}}{\bar{x}}  \right),
\end{equation}
induces a    Levi-flat foliation $\F$ tangent to $\G$ and, moreover, all tangent Levi-flat foliations are produced
in this manner. The comments in this paragraph can be summarized in the following statement:
\begin{prop}
\label{prop-saddle-node-pfaffian}
Let $\G$ be a germ of holomorphic foliation  at $0 \in \C^{2}$ of saddle-node type. Then  $\G$ is tangent to  a pfaffian hypersurface if and only if it is analytically normalizable.
\end{prop}

Let us consider a saddle-node singularity in its normal form \eqref{eq-saddle-node}. The holonomy of its strong separatrix $y=0$  with
respect to the section $\Sigma = \{(1,y); |y| < 1\}$ is defined
by the lifting along the leaves of the curve $\gamma(t) = (e^{2 \pi i t},0)$.
This holonomy map can be calculated  by taking  the flow at   time
$t=1$ of the real vector field defined by $2 \pi i g(y)$, where $g(y) = y^{k+1}/ (1 + \mu y^{k})$.
In coordinates $y = u + iv$,   this vector field is
\[ \frac{\pi }{2} \left( i (g(y) - \xbar{g(y)}) \frac{\dr}{\dr u} +     (g(y) + \xbar{g(y)}) \frac{\dr}{\dr v} \right) .\]
It defines, on $\Sigma$, a real-analytic foliation $\hat{\F}$, also given by the $1$-form
$\vartheta = \xbar{g(y)} \dd y  + g(y) \dd \bar{y}$.
On the other hand, from \eqref{eq-meromorphic-levi},   the Levi-flat foliation $\F$ restricted to $\Sigma$  is a one-dimensional
real analytic foliation $\F_{1}$
defined by
\[ \theta_{\alpha} = \frac{\alpha}{g(y)} \dd y + \frac{\bar{\alpha}}{\xbar{g(y)}} \dd \bar{y}. \]
We have that $\theta_{\alpha} \wedge  \vartheta  = (\alpha - \bar{\alpha}) \dd y \wedge \dd \bar{y}$.
Hence  $\F_{1}$ coincides with $\hat{\F}$ if and only if $\alpha \in \R$.
Leau-Fatou Flower Theorem says that the holonomy map of the strong separatrix has the structure of flower of $2k$ petals, corresponding to $k$ attracting directions
and $k$ repelling directions. This says that $\hat{\F}$ has a local
Bendixson sectorial decomposition  (see \cite[Th. 9.14]{ilyashenko2008})    formed by $2k$ elliptic sectors, separated by $2k$ parabolic sectors.
In particular, it is non-monodromic, saying that its leaves have the Rolle-Khovanskii property.
We should remark that if $\F_{1}$ and $\hat{\F}$ are different, then the leaves of $\F$ are not Rolle.
Indeed,   two non-coinciding  leaves of $\F_{1}$ and of $\hat{\F}$ intersecting at one point would   have infinitely many
points of intersection, bringing a contradiction with the finiteness Theorem.

Next we give an  example of a saddle-node singularity having an invariant   Rolle-pfaffian    hypersurface.
We will see in  Proposition \ref{prop-saddle-node-rolle} below   that this is the only case.
\begin{example}
\label{example-saddlenode-role}
{\rm
Suppose that, in \eqref{eq-saddle-node},  the weak multiplier satisfies $\mu \in \R$. Setting $\alpha = 1$ in \eqref{eq-meromorphic-levi}, we have
  \[\re(\tau)   =  \frac{1}{2} \left(
  \mu \frac{ \dd |y|^{2} }{|y|^{2}}
  - \frac{ \dd |x|^{2} }{|x|^{2}}
  -  \dd \left(  \frac{1}{k y^{k}}
  + \frac{1}{k \bar{y}^{k}} \right)
   \right)
  = \dd \left( \mu  \log |y|  -   \log |x| -    \re \left(\frac{1}{k y^{k}} \right) \right)
  \]
Hence, the function
\[ F(x,y) =  \mu  \log |y|  -   \log |x| -    \re \left(\frac{1}{k y^{k}} \right) ,\]
which is defined and real analytic outside the coordinate axes, is a first integral for the Levi-flat foliation $\F$ defined by
$\re(\tau)$, whose leaves are all
 Rolle pfaffian hypersurfaces tangent to $\G$.
 } \end{example}

Simple foliations of saddle-node type tangent to Rolle-pfaffian hypersurfaces are then characterized by the following statement:

\begin{prop}
\label{prop-saddle-node-rolle}
Let $\G$ be a holomorphic foliation having a simple saddle-node singularity at $0 \in \C^{2}$. Then    $\G$ is tangent to  a Rolle-pfaffian hypersurface if and only if both  properties below hold:
\begin{enumerate}
  \item it is analytically normalizable;
  \item its weak multiplier $\mu$ is real.
\end{enumerate}
\end{prop}
\begin{proof} We dealt with the question of analytic normalization in Proposition \ref{prop-saddle-node-pfaffian}.
It remains to consider the part of the statement concerning the weak multiplier. The  sufficiency  of $\mu \in \R$
is given by Example \ref{example-saddlenode-role}.
In order to prove the necessity, we suppose that $\mu \in \C \setminus \R$.
We are admitting that $\G$ is tangent to a pfaffian hypersurface, thus it is
tangent to a  Levi-flat foliation $\F$ defined, in normalizing holomorphic
coordinates $(x,y)$, by the real meromorphic $1$-form \eqref{eq-meromorphic-levi},
for some $\alpha \in \C^{*}$. We have seen above that if $\alpha \in \C \setminus \R$, then $\G$ is not tangent
to a Rolle-pfaffian hypersurface.
Thus, we can suppose that $\alpha = 1$.
Take $q = (0,1)$ in the weak separatrix and the section
$\Sigma = \{(x,1); |x| < 1\}$.
Then $\F$ restricted to $\Sigma$ is
defined by     $\bar{x} \dd x +  x  \dd \bar{x}$, which engenders a center type dynamics,
having as orbits circles centered at $0 \in \C$. Since the holonomy map of the weak separatrix   has $e^{2 \pi i \mu}$
as linear coefficient, any pfaffian leaf must contain infinitely many of these circles.
However, a Rolle-pfaffian leaf intersecting $\Sigma$ must consist of a single circle, otherwise
the Rolle condition would be violated, by taking in $\Sigma$, for instance,   a straight line through the origin as a curve  transverse to the foliation $\F$.
This  finishes the proof of the Proposition.
\end{proof}

\subsection{Simple resonant singularities}

The  only case not treated so far involves simple singularities with eigenvalue ratio $\lambda   \in \Q^{-}$.
These singularities are called \emph{resonant}.
Let us write $\lambda = \frac{p}{q}$, with $p,q \in \Z_{+}$ relatively prime.
If $\G$ is tangent to a germ of semi-analytic Rolle-pfaffian hypersurface, then the first integral
given by Cerveau-Lins Neto theorem must be holomorphic and $\G$ is analytically equivalent
to a linear $1$-form of the type
$\eta =  py  \dd x + qx \dd y$ \cite[Theorem ${\rm B}_{0}$]{mattei1980}.
Such  $\eta$ has the function $x^{p}y^{q}$ as a holomorphic first integral.

Let us then suppose  that $\G$ is tangent to a transcendent Rolle-pfaffian hypersurface. Hence,   Proposition \ref{prop-levi-flat-foliation}
gives that
$\G$ is tangent to a Levi-flat foliation.
By \cite[Lem. 1]{belko2003}, the holonomy of some of the separatrices of $\G$ is analytically normalizable,
implying that $\G$  itself is analytically normalizable \cite{martinet1983}. This means that  $\G$ is either defined by a linear
$1$-form as above  or
by a $1$-form of the type
\begin{equation}
\label{eq-resonant}
\eta = py(  1 +  (\mu -1) ( x^{p}y^{q})^{k} ) \dd x + q x(  1 + \mu (x^{p}y^{q})^{k} ) \dd y ,
\end{equation}
where  $k \geq 1$ and $\mu \in \C$.
This $\eta$ has an integrating factor, namely, the meromorphic $1$-form
\begin{equation}
\label{eq-resonant-closed}
 \tau = \frac{\eta}{xy(x^{p}y^{q})^{k}} =   (\mu-1) p \frac{ \dd x}{x} + \mu q \frac{\dd y}{y} - \dd \left( \frac{1}{k(x^py^q)^{k}} \right)
\end{equation}
is closed.

Let  $\mathcal F$ be the Levi-flat foliation defined by the real 1-form $\omega=\re(\tau)$.
The foliation $\mathcal F$ induces a real one-dimensional singular foliation   $\mathcal F_1$ on the disc  $\Sigma=\{(1,y)\colon |y|<r\}$, $r>0$. If we consider the parametrization
$\psi\colon y\mapsto (1,y)\in \Sigma$, the foliation $\mathcal F_1$ is defined by the real 1-form
$$\omega_1 = \psi^*(\omega) = \re \left\{ \left( \frac{\mu q}{y} + \frac{q}{y^{qk+1}} \right) \dd y \right\}.$$
From this we can see that  $\mathcal F_1$ is also defined by the real flow of the vector field
 $$X=i\frac {y^{qk+1}}{(1+\mu y^{qk})}\frac\partial{\partial y}.$$
The foliation  $\mathcal F_1$ has the multivalued first integral
$$\re\left( \mu q \log y -\frac 1{ky^{qk}}\right),$$
 which produces,
if  $\mu\in\mathbb R$, the real analytic first integral, for $y \neq 0$,
$$f(y) = \mu q\log |y|-\re\left( \frac 1{ky^{qk}}\right).$$
If $\mu \not\in \mathbb R$, we  can still find  a $C^\infty$  first integral
outside the coordinate axes, as we see next.
We will need a    lemma:
\begin{lem} If $r$ is small enough, there exists a $C^\infty$ submersion  $$f\colon\{0<|y|<r\}\to \mathbb R$$
with the following properties:
\begin{enumerate}
\item  $f$ is first integral of $\mathcal F_1$;
\item $f\left(y\right)=g(y^{qk})$ for some  $C^\infty$ submersion $g$.
\end{enumerate}
\end{lem}
\proof If we consider $\phi(y)=y^{qk}$,  we see that
$$X=\phi^*\left(iqk \frac {y^{2}}{(1+{\mu} y)}\frac\partial{\partial y}\right).$$
The vector field $iqk\frac {y^{2}}{(1+{\mu} y)}\frac\partial{\partial y}$ is analytically conjugated near the origin to
a vector field of the form   $X_{\lambda}= \frac {y^{2}}{(1+{\lambda} y)}\frac\partial{\partial y}$ for some $\lambda\in\mathbb C$. Moreover, it is known that the foliations near the origin defined by the real parts of two such vector fields  $X_{\lambda_1}$ and $X_{\lambda_2}$ are
equivalent by a homeomorphism that is a $C^\infty$ diffeomorphism outside the origin. In particular, the foliations near the origin of
the real parts of
$iqk\frac {y^{2}}{(1+{\mu} y)}\frac\partial{\partial y}$ and  $y^2\frac\partial{\partial y}$ are equivalent by a homeomorphism $\Phi$
that is a $C^\infty$ diffeomorphism outside the origin. Thus, composing $\Phi$  with the first integral
$\im (1/y)$ of
$\re \left\{ y^2\frac\partial{\partial y} \right\}$
we obtain a $C^\infty$ submersion $g=\im (1/\Phi)$ that is a  first integral of
$\re \left\{iqk\frac {y^{2}}{(1+{\mu} y)}\frac\partial{\partial y} \right\}$
in a punctured neighborhood of the origin.  Therefore  $f(y)=g(y^{qk})$ is a first integral of  $\mathcal F_1$.\qed

A $C^{\infty}$ first integral is then produced by extending the function $f$ along the leaves of $\F$:
\begin{prop} There exists a $C^\infty$ submersion that is a first integral  of $\mathcal F$ in the complement of
$\{xy=0\}$ in a neighborhood of
$0\in\mathbb C^2$.
\end{prop}
\proof Consider the set $U=\{0<|x|\le 1, 0<|y|< r\}$. If $r$ is small enough,
the holonomy of the separatrix $\{y=0\}$ can be computed in
$\Sigma$ and is given (see the Appendix) in the form $(1,y)\mapsto (1,h(y))$,  where
$$h=e^{-2\pi i (p/q)}\exp\left(2\pi i(p/q)\frac{y^{qk+1}}{1+\mu y^{qk}}\frac \partial{\partial y}\right)=
e^{-2\pi i (p/q)}\exp\left[2\pi  (p/q)X\right].$$
 By reducing $r$ if necessary, it is not difficult to prove that  the following properties hold:
\begin{enumerate}
\item The leaves of $\eta|_U=0$ intersect $\Sigma$.
\item If a leaf of $\eta|_U=0$ intersects $\Sigma$ at two points $(1,y)$ and $(1,y')$, there exists
$n\in\mathbb Z$ such that $y'=h^{\circ n}(y)$, where $h^{\circ n}$ denotes the $n$-th iteration of $h$.
\end{enumerate}
 Let  $\zeta\in U$ and let $L$  be the leaf
of the foliation $\eta|_U=0$ through $\zeta$. Then we choose any $(1,y)\in L\cap \Sigma$ and define    $F(\zeta)=f(y)$.
If well defined, it is clear that $F$  will be a  $C^\infty$ submersive  first integral of  $\mathcal F|_U$.  If $(1,y')$
is any other point in  $L\cap \Sigma$, we have that $y'=h^{\circ n}(y)$ for some  $n\in\mathbb Z$, so $F$ will be well defined if we show that $f\circ h=f$, which we do next. We have that
\begin{align*} f\circ h = g\circ h^{qk}=
g\circ \left[\exp\left(2\pi (p/q)X\right)\right]^{qk}\\
=f\circ\exp\left[2\pi (p/q)X\right]=f,
\end{align*}
because  $f$ is first integral of $X$. \qed

The existence of the first integral $f$, even if it is only of class $C^\infty$, guarantees that the leaves of  $\mathcal F|_U$ satisfy the property of  Rolle-Khovanskii.
Finally, the discussion above can be summarized in the following statement.
\begin{prop}
\label{prop-resonant-rolle}
Let $\mathcal G$  be a holomorphic foliation having a simple resonant non-lineariza\-ble singularity
at $0\in\mathbb C^2$. Then $\mathcal G$ is tangent to a Rolle-pfaffian hypersurface if and only if $\mathcal G$ is analytically
normalizable. In this case $\mathcal G$ is tangent to a Levi-flat foliation with Rolle-Pfaffian leaves  outside the separatrices of $\G$.
\end{prop}

\section{Some consequences}
\label{sec-consequences}

In this final section, our   objective is to obtain information on a germ of holomorphic foliation  $\G$ deriving from the assumption that it is
tangent to   Rolle-pfaffian hypersurfaces. In order to do that, however, we should establish a condition of tangency that works for our goals.
Let us write a few words of motivation.
In principle, since we are working with germs of holomorphic foliations,  it would be natural if we asked the origin to be an accumulation point of a Rolle-pfaffian hypersurface $V$.  Hence,  when lifting $V$ through the reduction of singularities $\pi: (\tilde{M},\D) \to (\C^{2},0)$ of $\G$, the  resulting surface $\tilde{V}$ would
accumulate in points contained in the exceptional divisor $\D$.
However, we cannot assure in general that this accumulation contains all invariant components of $\D$, even
if $\G$ is assumed to be non-dicritical. For instance,  if $\tilde{V}$  accumulates to
an elliptic singularity $p \in \D$ with positive multiplier --- $\tilde{V}$ is a nodal separator at $p$ --- then it does not accumulate in the components of $\D$ containing $p$ (see Example \ref{ex-nodal}). Another example: if $\tilde{V}$ accumulates to   a saddle-node singularity then it accumulates in  the strong separatrix but not in  the weak one, by the discussion carried out in
Subsection \ref{subsection-saddlenode}.
Thus, without further assumption, we cannot count on the fact that  $\tilde{V}$ ``spreads'' throughout the exceptional divisor $\D$.

On the other hand, it may occur that, even though $\G$ is tangent to a family of Rolle-pfaffian hypersurfaces arbitrarily near   $0 \in \C^{2}$, none
of them has the origin in its closure.
This is the case, for instance, if $\G$ is given by the level sets of a holomorphic function $f \in \cl{O}_{2}$:
it is tangent to each one of the pfaffian hypersurfaces $|f| = c$, with $c>0$.
This is also what happens for a linearizable elliptic singularity with negative multiplier.

In view of this discussion, we propose the following stronger condition of tangency:

\begin{ddef}{\rm
\label{def-Rolle-accumulation}
A germ of holomorphic foliation $\G$ at $(\C^{2},0)$ has \emph{Rolle tangency} if it is tangent to
 an integrable real analytic $1$-form $\omega$ defining a Rolle foliation.
}\end{ddef}

If $\G$ is tangent to a semi-analytic pfaffian hypersurface, then it has a meromorphic first integral, by
Theorem \ref{teo-cerveau-linsneto}. If this first integral is $f$, then $|f| = c$, with $c>0$, are leaves of
a Rolle foliation  invariant by $\G$ (defined by the $1$-form $\omega$ obtained
by cancelling the zeros and poles of $2 \re(\bar{f}\dd f) = \bar{f} \dd f + f \dd \bar{f}$). On the other hand,
if $\G$ is tangent to a transcendent pfaffian hypersurface $V$, then the real  $1$-form $\omega$ defining $V$
is integrable, inducing a Levi-flat foliation $\F$ tangent to $\G$. The condition of the definition says that $V$ and all solutions of $\omega$ are Rolle.

The discussion of Section \ref{sec-simples} allows us to identify simple singularities having Rolle
tangency:  hyperbolic singularities are excluded by Proposition
\ref{prop-hyperbolic}, elliptic singularities are linearizable by Proposition \ref{prop-elliptic}, resonant singularities
are analytically normalizable, by Proposition \ref{prop-resonant-rolle}, and saddle-node singularities are analytically normalizable and have real weak multiplier, by Proposition \ref{prop-saddle-node-rolle}.
Based on this discussion, we can then state
the following:

\begin{teo}
\label{teo-nondic-pfaffian}
Let $\G$ be a germ of   holomorphic foliation at $(\C^{2},0)$   having Rolle tangency.
Then $\G$ is defined by a closed meromorphic $1$-form and
the simple singularities of the reduction of singularities of $\G$ are of the following types:
\begin{enumerate}
  \item linearizable elliptic,
  \item  analytically normalizable resonant or
  \item  analytically normalizable saddle-nodes with real weak multiplier.
\end{enumerate}
\end{teo}
\begin{proof}
The  statement saying that $\G$ is defined by a closed meromorphic $1$-form is a consequence of Proposition \ref{prop-levi-flat-foliation} and of Theorems \ref{teo-cerveau-linsneto} and \ref{teo-Levi-closed}.
The hypothesis of $\G$ having the Rolle tangency says that all simple singularities of its reduction of singularities   also have  Rolle tangency.  The conclusion contained in items (1), (2) and (3) follows from the comments preceding the Theorem's statement.
 \end{proof}

A germ of holomorphic foliation $\G$ at $0 \in \C^{2}$ is \emph{logarithmic} if it is defined by a logarithmic $1$-form, i.e.
a closed meromorphic $1$-form with simple poles. This $1$-form may be written as $\tau$ in Equation \eqref{eq-closed-form},
removing the exact part.
A logarithmic foliation is  of \emph{generalized curve type}, i.e. belongs to the  family of germs of holomorphic foliations
without saddle-nodes in the
reduction of singularities   \cite{camacho1984}.
We say that $\G$ is \emph{logarithmic with real residues} if, in this writing, we
can obtain $\lambda_{i} \in \R$ for every $i=1,\ldots, k$. In this case, $\G$ is tangent to the Rolle foliation
defined by the level sets of $\varphi = |f_{1}|^{\lambda_{1}} \cdots |f_{k}|^{\lambda_{k}}$ and, according to our
definition, $0 \in \C^{2}$ has Rolle tangency. Moreover,  if we get 
$\lambda_{i} \in \Z$ for every $i=1,\ldots, k$, then $\G$ has a meromorphic first integral and the level sets of $\varphi$ are
semi-analytic. Otherwise, $\varphi$ defines transcendent Rolle-pfaffian hypersurfaces and, in this situation, $\G$ necessarily contains an elliptic singularity in its reduction of singularities.
The next theorem is a partial reciprocal for these facts:

\begin{teo}
\label{teo-logarithmic}
Let $\G$ be a germ of non-dicritical  holomorphic foliation at $(\C^{2},0)$  with  Rolle tangency.
Suppose that:
\begin{enumerate}
  \item $\G$ is of generalized curve type;
  \item $\G$ has at least one elliptic singularity in its reduction of singularities.
\end{enumerate}
 Then $\G$ is a logarithmic foliation with real residues.
\end{teo}
\begin{proof}
Denote by $\pi: (\tilde{M},\D) \to (\C^{2},0)$  the reduction of singularities of $\G$ and, for a component $D \subset \D$,
by $\holv(\tilde{\G},D) \subset \diff(\Sigma,p) \cong  \diff(\C,0)$, where $\tilde{\G} = \pi^{*}\G$,  the virtual holonomy group,
calculated with respect to a one dimensional section $\Sigma$ transverse to $\D$ at $p \in D$. Recall that this group consists of all diffeomorphisms
of $\diff(\Sigma,p)$ that preserve the sets of points of $\Sigma$ lying in the same leaf of $\tilde{\G}$.
We will prove that,    for every component $D \subset \D$, $\holv(\tilde{\G},D)$:
\begin{enumerate}[label=(\roman*)]
  \item   is analytically linearizable;
  \item   contains only rotations;
  \item necessarily contains an irrational rotation.
\end{enumerate}
Once we do that, we get the Theorem's conclusion   from \cite[Prop. 4.4]{mol-fernandez-rosas2020}.

We start with a fixed component $D \subset \D$ containing an elliptic singularity. The assumption of Rolle tangency implies that  the elliptic singularity is linearizable (Proposition \ref{prop-elliptic}) and, hence, it engenders in $\holv(\tilde{\G},D)$ an irrational rotation --- up to choosing an appropriate
complex coordinate in $\Sigma$. Therefore,
 the restriction of the Rolle foliation $\F$ to $\Sigma$ must be the real analytic foliation given by concentric
circles. This implies, in particular, that $\holv(\tilde{\G},D)$ contains no elements tangent to the identity (i.e. non-linearizable maps $\Phi \in \diff(\C,0)$ such that
$\Phi'(0) =1$). Indeed, if such an element existed, its flower structure
would imply that the intersection with $\Sigma$ of  each leaf of  $\F$ would contain infinitely many circles accumulating to $0 \in \C$,
contradicting the Rolle property.
This has two main consequences: first,
all resonant  singularities of $\tilde{\G} = \pi^{*} \G$ on $D$ are
linearizable, since, otherwise, their holonomy maps would be non-linearizable and their iteration would engender holonomy maps
tangent to the identity.
Thus,  using the assumption that  $\G$ is of generalized curve type, and also Propositions \ref{prop-hyperbolic} and \ref{prop-elliptic},
we have that the singularities of $\tilde{\G}$ in $D$ are all linearizable, either resonant or elliptic. Second, $\holv(\tilde{\G},D)$ is abelian, since
the commutator of a  pair of its elements, if non trivial,  would be tangent to the identity.
Hence,  a standard argument proves that $\holv(\tilde{\G},D)$ is analytically linearizable in the coordinate  that linearizes the
irrational rotation.

If $\D = D$, we are finished. Otherwise,   take a different component $D' \subset \D$ intersecting $D$ at a point $p$.
If $\tilde{\G}$ is
linearizable elliptic at $p$ then, evidently, $\holv(\tilde{\G},D')$ contains an irrational rotation.
On the other hand, if  $\tilde{\G}$ is
linearizable resonant at $p$, we can transfer to $\holv(\tilde{\G},D')$
the irrational rotation of  $\holv(\tilde{\G},D)$ by means of the Dulac correspondence (see \cite{scardua2016} for definition).
Hence, we can also conclude that all singularities of $\tilde{\G}$ on $D'$ are linearizable, either resonant or elliptic and that
$\holv(\tilde{\G},D')$ is analytically linearizable.
This argument can be iterated along the exceptional divisor, allowing us to conclude that
all virtual holonomy groups have the desired properties.
\end{proof}

We   observe that, in Theorems \ref{teo-nondic-pfaffian} and  \ref{teo-logarithmic}, the hypothesis of Rolle tangency cannot
be replaced by solely the existence of invariant Rolle-pfaffian hypersurfaces.
The
so-called \emph{logarithmic models} enables the construction of germs of logarithmic foliations in dimension two
with prescribed reduction of singularities and
\emph{Camacho-Sad indices}
 (i.e. the multiplipliers $\lambda$) of its simple singularities,
as long as some axiomatic conditions are fulfilled (see \cite{mol2021} and also \cite{corral2003,canocorral2006}). Thus, for instance,
it is possible to construct a logarithmic foliation $\G$ having simultaneously   nodal singularities ($\lambda \in \R_{+}$) and singularities of hyperbolic type ($\lambda \in \C \setminus \R$). The nodal singularities will produce nodal separators that
go down by the reduction of singularities map to invariant  hypersurfaces which are pfaffian,
since $\G$ is defined by a closed meromorphic $1$-form, and also Rolle, for their separating character.
However, the existence of hyperbolic singularities precludes the existence of a Rolle foliation tangent to $\G$.

\appendix
\section*{Appendix: holonomy of resonant normal singularities}
The holonomy of a resonant  singularity is computed by Martinet and Ramis \cite{martinet1983}. Although the map they give is correct up to linear change of coordinates, we need the exact expression of the holonomy computed in the transversal $\{1\}\times \mathbb C$. For this reason, even though its computation is standard, we opted to present it here.
We start by computing the holonomy of the separatrix $\{v=0\}$ of the singularity defined by the 1-form
$$\eta_k=v(1+(\mu-1)(uv)^k)du+u(1+\mu (uv)^k)dv.$$
This foliation is  the blow-up of the saddle-node given by the 1-form
$$\eta_{\text{sn}}= x(1+\mu y^k)dy-y^{k+1}dx,$$
where the blow-up map is $\pi(u,v)=(u,uv)$. Let $H_k\colon (1,v)\mapsto (1,h_k(v))$ be the holonomy of
 the separatrix $\{v=0\}$ of the foliation $\eta_k=0$  computed in a neighborhood of $(1,0)$ in $\Sigma=\{(1,v)\colon v\in\mathbb C\}$.
 Let  $\tilde H\colon (1,y)\mapsto (1,\tilde h(y))$ be the holonomy map of
 the separatrix $\{y=0\}$ of $\eta_{\text{sn}}=0$ computed in a neighborhood of $(1,0)$ in $\tilde \Sigma=\{(1,y)\colon y\in\mathbb C\}$.
 Then, since $\pi (\Sigma)= \tilde{\Sigma}$, we have that  $H_k=\pi^{-1}\circ \tilde H\circ \pi$, whence $h_k(v)=\tilde h (v)$. In order to compute $\tilde h$ we notice that the foliation $\eta_{\text{sn}}=0$
 is also generated by the vector field $$Z=2\pi ix\frac\partial{\partial x}+\frac{2\pi iy^{k+1}}{1+\mu y^k}  \frac\partial{\partial y}.$$
Then  we have
 $$(1,\tilde h (y))=(\exp Z)(1,y)=\left(1, \exp\left(\frac{2\pi iy^{k+1}}{1+\mu y^k}  \frac\partial{\partial y}\right) (y)\right).$$
 Therefore
 $$h_k=\exp\left(\frac{2\pi i v^{k+1}}{1+\mu v^k}  \frac\partial{\partial v}\right).$$

 Now, we deal with the resonant singularity defined by
  $$\eta=py(1+(\mu-1)(x^py^q)^k)dx
+qx(1+\mu(x^py^q)^k)dy,$$ for which we notice that $\eta=\sigma^*(\eta_{k})$ , where $\sigma(x,y)=(x^p,y^q)$.
 Let  $H\colon (1,y)\mapsto (1,h(y))$ be the holonomy map of
 the separatrix $\{y=0\}$ of $\eta=0$ computed in a neighborhood of $(1,0)$ in $\{(1,y)\colon y\in\mathbb C\}$.
 As above, let $H_k\colon (1,v)\mapsto (1,h_k(v))$ be the holonomy of
 the separatrix $\{v=0\}$ of the foliation $\eta_k=0$  computed in a neighborhood of $(1,0)$ in $\{(1,v)\colon v\in\mathbb C\}$.
 Since the loop $(e^{2\pi i t},0)$, $t\in[0,1]$, is mapped by $\sigma$ onto the loop $(e^{2p\pi i t},0)$, $t\in[0,1]$,  the holonomy maps $H$ and $H_k$ are related by the equation
 $\sigma\circ H=H_k^{\circ p}\circ \sigma$ and, hence,
  $$(h(y))^q=h_k^{\circ p}(y^q)=
 \exp\left(\frac{2p\pi i v^{k+1}}{1+\mu v^k}  \frac\partial{\partial v}\right)(y^q).$$
 Notice that, if  $\Psi(y)=y^q$,
 $$ \Psi^*\left(\frac{2p\pi i v^{k+1}}{1+\mu v^k}  \frac\partial{\partial v}\right)=
 2\pi i(p/q) \frac{ y^{qk+1}}{1+\mu y^{qk}}  \frac\partial{\partial y}. $$
 Then
 $$\exp\left(\frac{2p\pi i v^{k+1}}{1+\mu v^k}  \frac\partial{\partial v}\right)(y^q)=
\left[
 \exp\left(
 2\pi i(p/q) \frac{ y^{qk+1}}{1+\mu y^{qk}}  \frac\partial{\partial y}\right)(y)\right]^q $$
 and, therefore, we obtain that
 $$h= \xi\exp\left(
 2\pi i(p/q) \frac{ y^{qk+1}}{1+\mu y^{qk}}  \frac\partial{\partial y}\right),$$
where $\xi^q=1$. At this point we conclude that $\xi =e^{-2\pi i(p/q)}$, because the Camacho-Sad index of $\eta=0$ along the separatrix $\{y=0\}$ is equal to $-p/q$.

\bibliographystyle{plain}
\bibliography{referencias}

\def\cprime{$'$} \def\polhk#1{\setbox0=\hbox{#1}{\ooalign{\hidewidth
  \lower1.5ex\hbox{`}\hidewidth\crcr\unhbox0}}}
\begin{thebibliography}{10}

\bibitem{belko2003}
D.~Belko~Garba.
\newblock Caract\'erisation des feuilletages holomorphes singuliers, contenus
  dans des feuilletages {L}evi flat, sur les surfaces complexes compactes.
\newblock {\em Bull. Sci. Math.}, 127(10):845--857, 2003.

\bibitem{bracci2010}
F.~Bracci.
\newblock Local holomorphic dynamics of diffeomorphisms in dimension one.
\newblock In {\em Five lectures in complex analysis}, volume 525 of {\em
  Contemp. Math.}, pages 1--42. Amer. Math. Soc., Providence, RI, 2010.

\bibitem{mol2021}
J.~Bretas and R.~Mol.
\newblock Logarithmic models and meromorphic functions in dimension two.
\newblock {\em Pre-print}, 2021, https://arxiv.org/abs/2110.07637.

\bibitem{camacho1984}
C.~Camacho, A.~Lins~Neto, and P.~Sad.
\newblock Topological invariants and equidesingularization for holomorphic
  vector fields.
\newblock {\em J. Differential Geom.}, 20(1):143--174, 1984.

\bibitem{camacho1982}
C.~Camacho and P.~Sad.
\newblock Invariant varieties through singularities of holomorphic vector
  fields.
\newblock {\em Ann. of Math. (2)}, 115(3):579--595, 1982.

\bibitem{canocerveau2013}
F.~Cano, D.~Cerveau, and J.~D\'{e}serti.
\newblock {\em Th\'{e}orie \'{e}l\'{e}mentaire des feuilletages holomorphes
  singuliers}.
\newblock Echelle. Belin, 2013.

\bibitem{canocorral2006}
F.~Cano and N.~Corral.
\newblock Dicritical logarithmic foliations.
\newblock {\em Publ. Mat.}, 50(1):87--102, 2006.

\bibitem{cerveau2011}
D.~Cerveau and A.~Lins~Neto.
\newblock Local {L}evi-flat hypersurfaces invariants by a codimension one
  holomorphic foliation.
\newblock {\em Amer. J. Math.}, 133(3):677--716, 2011.

\bibitem{cerveau1982}
D.~Cerveau and J.-F. Mattei.
\newblock {\em Formes int\'egrables holomorphes singuli\`eres}, volume~97 of
  {\em Ast\'erisque}.
\newblock Soci\'et\'e Math\'ematique de France, Paris, 1982.
\newblock With an English summary.

\bibitem{cerveau2004}
D.~Cerveau and P.~Sad.
\newblock Fonctions et feuilletages {L}evi-flat. \'etude locale.
\newblock {\em Ann. Sc. Norm. Super. Pisa Cl. Sci. (5)}, 3(2):427--445, 2004.

\bibitem{corral2003}
N.~Corral.
\newblock Sur la topologie des courbes polaires de certains feuilletages
  singuliers.
\newblock {\em Ann. Inst. Fourier (Grenoble)}, 53(3):787--814, 2003.

\bibitem{mol-fernandez-rosas2020}
A.~Fern\'{a}ndez-P\'{e}rez, R.~Mol, and R.~Rosas.
\newblock On singular real analytic {L}evi-flat foliations.
\newblock {\em Asian J. Math.}, 24(6):1007--1028, 2020.

\bibitem{ilyashenko2008}
Y.~Ilyashenko and S.~Yakovenko.
\newblock {\em Lectures on analytic differential equations}, volume~86 of {\em
  Graduate Studies in Mathematics}.
\newblock American Mathematical Society, Providence, RI, 2008.

\bibitem{khovanskii1984}
A.~Khovanski\u{\i}.
\newblock Real analytic manifolds with the property of finiteness, and complex
  abelian integrals.
\newblock {\em Funktsional. Anal. i Prilozhen.}, 18(2):40--50, 1984.

\bibitem{khovanskii1991}
A.~Khovanski\u{\i}.
\newblock {\em Fewnomials}, volume~88 of {\em Translations of Mathematical
  Monographs}.
\newblock American Mathematical Society, Providence, RI, 1991.
\newblock Translated from the Russian by Smilka Zdravkovska.

\bibitem{marin2012}
D.~Mar\'{\i}n and J.-F. Mattei.
\newblock Monodromy and topological classification of germs of holomorphic
  foliations.
\newblock {\em Ann. Sci. \'{E}c. Norm. Sup\'{e}r. (4)}, 45(3):405--445, 2012.

\bibitem{martinet1982}
J.~Martinet and J.-P. Ramis.
\newblock Probl\`emes de modules pour des \'equations diff\'erentielles non
  lin\'eaires du premier ordre.
\newblock {\em Inst. Hautes \'Etudes Sci. Publ. Math.}, (55):63--164, 1982.

\bibitem{martinet1983}
J.~Martinet and J.-P. Ramis.
\newblock Classification analytique des \'{e}quations diff\'{e}rentielles non
  lin\'{e}aires r\'{e}sonnantes du premier ordre.
\newblock {\em Ann. Sci. \'{E}cole Norm. Sup. (4)}, 16(4):571--621 (1984),
  1983.

\bibitem{mattei1980}
J.-F. Mattei and R.~Moussu.
\newblock Holonomie et int\'egrales premi\`eres.
\newblock {\em Ann. Sci. \'Ecole Norm. Sup. (4)}, 13(4):469--523, 1980.

\bibitem{moussu-roche1992}
R.~Moussu and C.~Roche.
\newblock Th\'{e}or\`emes de finitude pour les vari\'{e}t\'{e}s pfaffiennes.
\newblock {\em Ann. Inst. Fourier (Grenoble)}, 42(1-2):393--420, 1992.

\bibitem{scardua2016}
B.~Sc\'{a}rdua.
\newblock Simple dynamics and integrability for singularities of holomorphic
  foliations in dimension two.
\newblock {\em J. Singul.}, 14:148--171, 2016.

\bibitem{seidenberg1968}
A.~Seidenberg.
\newblock Reduction of singularities of the differential equation
  {$A\,dy=B\,dx$}.
\newblock {\em Amer. J. Math.}, 90:248--269, 1968.

\end{thebibliography}

%
%
%
%
%
%

\end{document}